\documentclass[letterpaper, 10pt, conference, fleqn]{ieeeconf}
\IEEEoverridecommandlockouts
\usepackage{times} % assumes new font selection scheme installed
\usepackage{amsmath} % assumes amsmath package installed
\usepackage{amssymb}  % assumes amsmath package installed
\usepackage{mathtools}
\usepackage{amsfonts}
\usepackage{amssymb}
\usepackage{color}
\usepackage{graphicx}

\newtheorem{Theorem}{Theorem}
\newtheorem{Lemma}{Lemma}
\newtheorem{Proposition}{Proposition}
\newtheorem{Definition}{Definition}
\title{\LARGE \bf
    On Dual of LMIs for Absolute Stability Analysis of \\
    Nonlinear Feedback Systems with \\
    Static O'Shea-Zames-Falb Multipliers
}

\author{Hibiki Gyotoku, Tsuyoshi Yuno, Yoshio Ebihara, Victor Magron, Dimitri Peaucelle, and Sophie Tarbouriech% <-this % stops a space
    \thanks{
        Hibiki Gyotoku is with the Graduate School of Information Science and Electrical Engineering, Kyushu University, 
        Fukuoka 8190395, Japan 
        {\tt\small gyotoku.hibiki.327@s.kyushu-u.ac.jp}.   %
        Tsuyoshi Yuno and Yoshio Ebihara are with the Faculty of Information Science and Electrical Engineering, Kyushu University, 
        Fukuoka 8190395, Japan
        {\tt\small yuno@cig.ees.kyushu-u.ac.jp; ebihara@ees.kyushu-u.ac.jp}.  %
V. Magron, D. Peaucelle, and S. Tarbouriech are
LAAS-CNRS, Universit\'{e} de Toulouse, CNRS, Toulouse, France.  
This work was supported by JSPS KAKENHI Grant Number JP21H01354.  
This work was also supported by the AI Interdisciplinary Institute ANITI funding, through the French "Investingfor the Future PIA3" program under the Grant agreement n°ANR-19-PI3A-0004 as well as by the National Research Foundation, Prime Minister's Oﬃce, Singapore under its Campus for Research Excellence and Technological Enterprise (CREATE) programme. }
}

\newcommand{\nc}{\newcommand}
\nc{\tbf}[1]{\textbf{#1}}
\nc{\bb}[1]{\mathbb{#1}}
\nc{\mrm}[1]{\mathrm{#1}}
\nc{\mbf}[1]{\mathbf{#1}}
\nc{\cl}[1]{\mathcal{#1}}
\nc{\projection}[2]{\cl{P}_{\mrm{#1}}\left(#2\right)}
\nc{\1}{\mathbf{1}_m}
\nc{\slope}[2]{\mrm{slope}[#1,#2]}
\nc{\trace}{\mrm{trace}}
\nc{\He}{\mrm{He}}
\nc{\rank}[1]{\mrm{rank}(#1)}

\nc{\PLMIfirst}%
{
    \begin{bmatrix}
        PA+A^TP&PB\\
        B^TP&0
    \end{bmatrix}
}
\nc{\PLMIDHDsecond}%
{
    (\ast)^{T} \cl{M}(M,\mu,\nu) \begin{bmatrix}
        C & D \\
        0 & I_m
    \end{bmatrix}
}
\nc{\PLMIDDsecond}%
{
    (\ast)^{T} \cl{M}(M_{\mrm{d}}+M_{\mrm{od}},\mu,\nu) \begin{bmatrix}
        C & D \\
        0 & I_m
    \end{bmatrix}
}
\nc{\RePLMIDHDsecond}%
{
    (\ast)^{T} \cl{M}(M,0,1) \begin{bmatrix}
        C & D \\
        0 & I_m
    \end{bmatrix}
}
\nc{\RePLMIDDsecond}%
{
    (\ast)^{T} \cl{M}(M_{\mrm{d}}+M_{\mrm{od}},0,1) \begin{bmatrix}
        C & D \\
        0 & I_m
    \end{bmatrix}
}
\nc{\matrixY}%
{
    \begin{bmatrix}
        - \mu C&- \mu D+I_m
    \end{bmatrix} H \begin{bmatrix}
        \nu C^T\\
        \nu D^T-I_m
    \end{bmatrix}
}
\nc{\RematrixY}%
{
    \begin{bmatrix}
        0_{m,n}&I_m
    \end{bmatrix} H \begin{bmatrix}
        C^T\\
        D^T-I_m
    \end{bmatrix}
}
\begin{document}
\maketitle
\thispagestyle{empty}
\pagestyle{empty}

%%%%%%%%%%%%%%%%%%%%%%%%%%%%%%%%%%%%%%%%%%%%%%%%%%%%%%%%%%%%%%%%%%%%%%%%%%%%%
%%%%%%%%%%%%%%%%%%%%%%%%%%%%%%%%%%%%%%%%%%%%%%%%%%%%%%%%%%%%%%%%%%%%%%%%%%%%%%%%
%
\begin{abstract}
    This study investigates the absolute stability criteria 
    based on the framework of integral quadratic constraint (IQC) 
    for feedback systems with slope-restricted nonlinearities.  
    In existing works, well-known absolute stability certificates 
    expressed in the IQC-based linear matrix inequalities (LMIs) were derived, 
    in which the input-to-output characteristics of the slope-restricted nonlinearities 
    were captured through static O'Shea-Zames-Falb multipliers. 
    However, since these certificates are only sufficient conditions, 
    they provide no clue about the absolute stability in the case where the LMIs are infeasible. 
    In this paper, by taking advantage of the duality theory of LMIs, 
    we derive a condition for systems to be \emph{not} absolutely stable 
    when the above-mentioned LMIs are infeasible. 
    In particular, we can identify a destabilizing nonlinearity within 
    the assumed class of slope-restricted nonlinearities as well as 
    a non-zero equilibrium point of the resulting closed-loop system, 
    by which the system is proved to be not absolutely stable. 
    We demonstrate the soundness of our results by numerical examples.
\end{abstract}

%%%%%%%%%%%%%%%%%%%%%%%%%%%%%%%%%%%%%%%%%%%%%%%%%%%%%%%%%%%%%%%%%%%%%%%%%%%%%
%%%%%%%%%%%%%%%%%%%%%%%%%%%%%%%%%%%%%%%%%%%%%%%%%%%%%%%%%%%%%%%%%%%%%%%%%%%%%%%%
%
\section{Introduction}
Recently, there has been a growing attention to control theoretic approaches to 
the analysis of optimization algorithms \cite{Lessard_SIAM2016}, 
stability analysis of dynamic neural networks (NNs) 
\cite{Revay_LCSS2021,Ebihara_EJC2021,Ebihara_CDC2021}, 
and performance analysis of control systems driven by NNs 
\cite{Yin_IEEE2022, Scherer_IEEEMag2022,Souza_Automatica2023}.  
These algorithms, NNs and control systems can be 
modeled as feedback systems consisting 
of linear time-invariant (LTI) systems and nonlinear operators, 
and thus control theory can be applied to the analyses of these objects \cite{Tarbouriech_2011}. 
In particular, since NNs have a large number of various nonlinear activation operators 
such as hyperbolic tangent (tanh), sigmoid, and rectified linear unit (ReLU), 
it is required to establish a method of analyzing feedback systems 
with a variety of nonlinear operators.

Against this background, we tackle the problem of absolute stability analysis 
of nonlinear feedback systems. 
Here, feedback systems are said to be absolutely stable 
if their origins are globally asymptotically stable for all nonlinear operators that belong to the assumed class  
\cite{Khalil_2002}.
In particular, in this paper, we focus on systems 
with slope-restricted and repeated nonlinear operators \cite{Valmorbida_IEEE2018,Fagundes_SCL2024}.  
O'Shea-Zames-Falb (OZF) multipliers \cite{O'Shea_IEEE1967,Zames_SIAM1968} 
are known to be effective for capturing 
the input-to-output characteristics of slope-restricted nonlinearities \cite{Carrasco_EJC2016}. 
By using OZF multipliers in the framework of Integral Quadratic Constraint (IQC)
\cite{IQC}, 
we can drive linear matrix inequality (LMI) conditions to ensure the absolute stability of the feedback systems.
However, since these LMI conditions are only sufficient conditions, 
we can conclude nothing about the absolute stability when the LMIs are infeasible.

To address this issue, we focus on the dual LMIs \cite{Scherer_EJC2006}
of these IQC-based LMIs.
As the main result, 
we show that, 
if the solution of the dual LMI satisfies a certain rank condition, then 
1) it is possible to extract a nonlinear operator, 
within the assumed class of slope-restricted nonlinearities, 
that destabilizes the target feedback system 
2) the resulting closed-loop system with the extracted nonlinearity has 
a non-zero equilibrium point whose value can be explicitly identified, 
by which the closed-loop system is proved to be not globally asymptotically stable, 
3) and hence the feedback system is not absolutely stable. 
We illustrate the effectiveness of the main result by numerical examples.  
We finally note that, in our recent paper \cite{Yuno_MICNON2024}, 
we also considered the dual of IQC-based LMIs to obtain instability certificates
for nonlinear feedback systems.  
However, the nonlinearities there were fixed to be ReLUs.  
In the present paper, we show that we can extract 
a destabilizing nonlinearity within the assumed class of slope-restricted nonlinearities
as stated above, and this clearly distinguishes the current contribution
from that of  \cite{Yuno_MICNON2024}.  

Notation: 
The set of $n \times m$ real matrices is denoted by $\bb{R}^{n \times m}$. 
We denote 
the $n \times n$ identity matrix, the $n \times n$ zero matrix, 
and the $n \times m$ zero matrix 
by $I_n$, $0_n$, $0_{n,m}$, respectively. 
For a matrix $A$, we write $A\ge 0$ to denote that $A$ is entrywise nonnegative. 
We denote the set of $n \times n$ real symmetric, positive semidefinite, and positive definite matrices by 
$\bb{S}^n$, $\bb{S}_+^n$, and $\bb{S}_{++}^n$, respectively. 
For $A\in\bb{S}^n$, we write $A\succ 0 \ (A\prec 0)$ to denote 
that $A$ is positive (negative) definite. 
For $A\in\bb{R}^{n \times n}$ and $B\in\bb{R}^{n \times m}$, 
$(\ast)^{T}AB$ is a shorthand notation of $B^{T}AB$, and we also define
$\He\{A\}:=A+A^T$. 
For $v\in\bb{R}^n$, we denote by $\|v\|$ its standard Euclidean norm. 
The induced norm of a (possibly nonlinear) operator $\Phi:\bb{R}^m \to \bb{R}^n$ 
is defined by 
$\|\Phi\|:=\sup_{v\in\bb{R}^m \backslash \{0\}}\dfrac{\|\Phi(v)\|}{\|v\|}$. 
For $A\in\bb{R}^{n \times n}$, we define $|A|_{\mrm{d}}\in\bb{R}^{n \times n}$ 
by $|A|_{\mrm{d},i,i}=A_{i,i}$ and $|A|_{\mrm{d},i,j}=-|A_{i,j}| \ (i \ne j)$. 
We also define 
\begin{align*}
    &
    \bb{D}^{m}:=\left\{ M\in\bb{R}^{m \times m}:M_{i,j}=0 \ (i \ne j, \ i,j=1,\ldots,m) \right\}, 
    \\
    &
    \bb{OD}^{m}:=\left\{ M\in\bb{R}^{m \times m}:M_{i,i}=0 \ (i=1,\ldots,m) \right\}.
\end{align*}
We finally define $\cl{P}_{\mrm{d}}:\bb{R}^{m \times m} \to \bb{D}^{m}$ 
as the projection onto $\bb{D}^m$, 
and $\cl{P}_{\mrm{od}}:\bb{R}^{m \times m} \to \bb{OD}^{m}$ 
as the projection onto $\bb{OD}^m$. 

%%%%%%%%%%%%%%%%%%%%%%%%%%%%%%%%%%%%%%%%%%%%%%%%%%%%%%%%%%%%%%%%%%%%%%%%%%%%%
%%%%%%%%%%%%%%%%%%%%%%%%%%%%%%%%%%%%%%%%%%%%%%%%%%%%%%%%%%%%%%%%%%%%%%%%%%%%%%%%
%
\section{IQC and Static O'Shea-Zames-Falb Multipliers}
%%%%%%%%%%%%%%%%%%%%%%%%%%%%%%%%%%%%%%%%%%%%%%%%%%%%%%%%%%%%%%%%%%%%%%%%%%%%%
\subsection{Basic Stability Conditions Based on IQC}
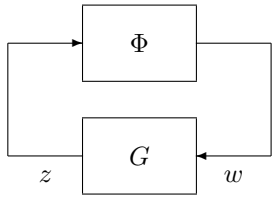
\begin{figure}[b]
    \centering
    \begin{picture}(3.5,2.5)(0,0)
\put(0,2){\vector(1,0){1}}
\put(1,1.5){\framebox(1.5,1){$\Phi$}}
\put(2.5,2){\line(1,0){1}}
\put(3.5,2){\line(0,-1){1.5}}
\put(3.5,0.5){\vector(-1,0){1}}
\put(3,0.3){\makebox(0,0)[t]{$w$}}
\put(1,0){\framebox(1.5,1){$G$}}
\put(1,0.5){\line(-1,0){1}}
\put(0.5,0.3){\makebox(0,0)[t]{$z$}}
\put(0,0.5){\line(0,1){1.5}}

\end{picture}
    \caption{Nonlinear Feedback System $\Sigma$}
    \label{fig:system}
\end{figure}
Let us consider the feedback system $\Sigma$ shown in Fig. \ref{fig:system}. 
Here, $G$ is a linear system described by
\begin{align}
    G:
    \begin{cases}
        \dot{x}(t)=Ax(t)+Bw(t) ,\\
        z(t)=Cx(t)+Dw(t), 
    \end{cases}
    \label{eq:systemG}
\end{align}
where $A\in\bb{R}^{n \times n}$, $B\in\bb{R}^{n \times m}$, $C\in\bb{R}^{m \times n}$, 
and $D\in\bb{R}^{m \times m}$.
We assume that $A$ is Hurwitz stable. 
On the other hand, 
$\Phi:\bb{R}^{m}\to\bb{R}^{m}$ stands for a static nonlinearity described by
\begin{align}
    w(t)=\Phi(z(t)).
    \label{eq:nonlinear}
\end{align}
We also assume that the feedback system $\Sigma$ is well-posed. 

In this paper, we are interested in the stability analysis 
of the feedback system $\Sigma$ in the framework of IQC with multipliers. 
For simplicity, we say that the feedback system $\Sigma$ is stable 
if the origin $x=0$ is globally asymptotically stable. 
The next proposition forms the basis 
of the stability analysis in the framework of IQC. 
\begin{Proposition}[\cite{IQC}]\label{pr:IQC_basic}
    Let us define $\mathbf{\Pi^{\star}}\subset\bb{S}^{2m}$ by
    \begin{align}
        \mathbf{\Pi^{\star}}:=\Biggl\{ \Pi\in\bb{S}^{2m} {}:{}
        (\ast)^T \Pi \begin{bmatrix}
            \zeta \\ 
            \Phi(\zeta)
        \end{bmatrix} \ge 0 \ \forall{\zeta}\in\bb{R}^m \Biggr\}. 
        \label{eq:Pi_star}
    \end{align}%
    Then, the system $\Sigma$ is stable 
    if there exist $P \succ 0$ and $\Pi \in \mathbf{\Pi^{\star}}$ such that
    \begin{align}
        \PLMIfirst+(\ast)^T \Pi \begin{bmatrix}
            C&D \\ 
            0&I_m
        \end{bmatrix} \prec 0.
        \label{eq:PLMI_basic}
    \end{align}%
\end{Proposition}

In IQC-based stability conditions such as Proposition \ref{pr:IQC_basic}, 
it is of prime importance 
to employ a set of multipliers $\mathbf{\Pi} \subset \mathbf{\Pi}^{\star}$ 
that allows us to reduce \eqref{eq:PLMI_basic} into a numerically tractable LMI problem 
and can capture the input-output properties of $\Phi$ as accurately as possible. 
In this paper, we focus
on static OZF multipliers for slope-restricted nonlinearities.

%%%%%%%%%%%%%%%%%%%%%%%%%%%%%%%%%%%%%%%%%%%%%%%%%%%%%%%%%%%%%%%%%%%%%%%%%%%%%
\subsection{Static O'Shea-Zames-Falb Multipliers}
Some specific definitions are necessary to describe static OZF multipliers. 
A matrix $M\in\bb{R}^{m \times m}$ is said to be Z-matrix 
if $M_{i,j} \le 0$ for all $i \ne j$. 
Moreover, $M$ is said to be doubly hyperdominant 
if it is a Z-matrix and $M\1 \ge 0$, $\1^{T}M \ge 0$, 
where $\1\in\bb{R}^m$ stands for the all-ones-vector. 
In addition, $M$ is said to be doubly dominant 
if $|M|_{\mrm{d}}\1 \ge 0$, $\1^{T}|M|_{\mrm{d}}$. 
In this paper, we denote by $\bb{Z}^m, \bb{DHD}^m, \bb{DD}^m \subset \bb{R}^{m \times m}$ 
the sets 
of Z-matrices, doubly hyperdominant matrices, and doubly dominant matrices, respectively. 
Namely, we define
\begin{align}
    \begin{aligned}
        &\bb{DHD}^m:=\{ M\in\bb{Z}^m:M\1 \ge 0, \1^{T}M \ge 0 \}, \\
        &\bb{DD}^m:=\{ M\in\bb{R}^{m \times m}:|M|_{\mrm{d}}\1 \ge 0, \1^{T}|M|_{\mrm{d}} \ge 0 \}.
    \end{aligned}
    \label{eq:DHD-DD_def}
\end{align}
It is obvious that $\bb{DHD}^m \subsetneq \bb{DD}^m$.

As previously stated, we focus on slope-restricted nonlinearities in this paper. 
\begin{Definition}\label{de:slope-restricted}
    Let $\mu \le 0 \le \nu$. 
    Then, a nonlinearity $\phi:\bb{R}\to\bb{R}$ is said to be slope-restricted, 
    in short $\phi\in\slope \mu  \nu $, if $\phi(0)=0$ and 
    \begin{align}
        \mu \le \frac{\phi(p)-\phi(q)}{p-q} \le \nu \ (\forall{p,q}\in\bb{R}, \ p \ne q).
        \notag
    \end{align}
\end{Definition}

Let us define $\mrm{diag}_{m}(\phi):\bb{R}^{m}\to\bb{R}^{m}$ by 
\begin{align}
    \mrm{diag}_{m}(\phi):=\mrm{diag}(\underbrace{\phi,\ldots,\phi}_{m}). 
    \label{eq_diag}
\end{align}
We further make the following definition.  
\begin{Definition}
    For $\mu \le 0 \le \nu$, we define 
    \begin{align}
        &
        \mathbf{\Phi}_{\mu,\nu}^{m}:=\{ \Phi:\Phi=\mrm{diag}_{m}(\phi),  \ \phi\in\slope \mu  \nu  \}, 
        %\label{set:slope-restricted} \\
        \notag \\
        &
        \mathbf{\Phi}_{\mrm{odd}}^{m}:=\{ \Phi:\Phi=\mrm{diag}_{m}(\phi), \ \phi \text{ is odd} \}.
        %\label{set:odd}
        \notag
    \end{align}
\end{Definition}

Under these definitions, the main result of \cite{OZFbasicresult} 
on static OZF multipliers \cite{O'Shea_IEEE1967,Zames_SIAM1968}
for slope-restricted nonlinearities can be summarized by the next lemma.
\begin{Lemma}[\cite{OZFbasicresult}]\label{lemma:OZF_basic}
    Let $\mu \le 0 \le\nu$. 
    Then, for all $\Phi\in\mathbf{\Phi}_{\mu,\nu}^{m}$ and $M\in\bb{DHD}^{m}$, we have
    \begin{align}
        \begin{aligned}
            &
            (\ast)^{T} \cl{M}(M,\mu,\nu)\begin{bmatrix}
                \zeta \\
                \Phi(\zeta)
            \end{bmatrix} \ge 0 \ \forall{\zeta}\in\bb{R}^m, \\
            &
            \cl{M}(M,\mu,\nu):=(\ast)^{T}\begin{bmatrix}
                0_m & M \\
                M^T & 0_m
            \end{bmatrix} \begin{bmatrix}
                \nu I_m & -I_m \\
                - \mu I_m & I_m
            \end{bmatrix}.
        \end{aligned}
        %\label{eq:OZF_basic}
        \notag
    \end{align}
    Moreover, this result holds 
    for all $\Phi\in\mathbf{\Phi}_{\mu,\nu}^{m}\cap\mathbf{\Phi}_{\mrm{odd}}^{m}$ 
    and $M\in\bb{DD}^m$. 
\end{Lemma}

On the basis of this lemma, we define 
\begin{align}
    &
    \mathbf{\Pi}_{m,\bb{DHD}}:= 
    \left\{ \Pi\in\bb{S}^{2m}:\Pi=\cl{M}(M,\mu,\nu), \ M\in\bb{DHD}^m \right\},
    %\label{set:pi_DHD} \\
    \notag \\
    &
    \mathbf{\Pi}_{m,\bb{DD}}:= 
    \left\{ \Pi\in\bb{S}^{2m}:\Pi=\cl{M}(M,\mu,\nu), \ M\in\bb{DD}^m \right\}.
    %\label{set:pi_DD}
    \notag
\end{align} 
Here, $\mathbf{\Pi}_{m,\bb{DHD}}$ is used if $\Phi\in\mathbf{\Phi}_{\mu,\nu}^{m}$ 
and $\mathbf{\Pi}_{m,\bb{DD}}$ is used 
if $\Phi\in\mathbf{\Phi}_{\mu,\nu}^{m}\cap\mathbf{\Phi}_{\mrm{odd}}^{m}$. 

%%%%%%%%%%%%%%%%%%%%%%%%%%%%%%%%%%%%%%%%%%%%%%%%%%%%%%%%%%%%%%%%%%%%%%%%%%%%%
%%%%%%%%%%%%%%%%%%%%%%%%%%%%%%%%%%%%%%%%%%%%%%%%%%%%%%%%%%%%%%%%%%%%%%%%%%%%%%%%
%
\section{Main Results for Slope-Restricted Nonlinearity $\Phi\in\mbf{\Phi}_{\mu,\nu}^{m}$}
%%%%%%%%%%%%%%%%%%%%%%%%%%%%%%%%%%%%%%%%%%%%%%%%%%%%%%%%%%%%%%%%%%%%%%%%%%%%%
\subsection{LMIs Ensuring Absolute Stability and Their Dual}
We have next lemma from Proposition \ref{pr:IQC_basic} and Lemma \ref{lemma:OZF_basic}.
\begin{Lemma}\label{lemma:DHD}
    Let $\mu \le 0 \le \nu$. 
    Then, the system $\Sigma$ consisting of \eqref{eq:systemG} and \eqref{eq:nonlinear}
    is absolutely stable for $\Phi\in\mbf{\Phi}_{\mu,\nu}^{m}$ 
    if there exist $P\in\bb{S}_{++}^n$ and $M\in\bb{DHD}^m$ 
    such that 
    \begin{align}
        \PLMIfirst + \PLMIDHDsecond \prec 0.  %\hspace*{-10mm}
        \label{eq:lem_DHD}
    \end{align}
\end{Lemma}

If LMI (\ref{eq:lem_DHD}) is numerically feasible, we can readily conclude 
from Lemma \ref{lemma:DHD} 
that the feedback system $\Sigma$ is absolutely stable for $\Phi\in\mbf{\Phi}_{\mu,\nu}^{m}$. 
However, this LMI condition is generally a sufficient condition 
and hence if the LMI turns out to be infeasible we can conclude nothing 
about the absolute stability of the system $\Sigma$.  
To address this issue, we derive the dual of LMI (\ref{eq:lem_DHD}). 
To that end, we first note that the LMI (\ref{eq:lem_DHD}) can be rewritten, 
equivalently, as follows: \\
\tbf{Primal LMI (For $\Phi\in\mbf{\Phi}_{\mu,\nu}^{m}$)} \\  
Find $P\in\bb{S}_{++}^n$ and $M\in\bb{R}^{m \times m}$ 
such that 
\begin{align}
    \scalebox{0.95}{$
        \begin{aligned}
            &
            \PLMIfirst + \PLMIDHDsecond \prec 0, \hspace*{-5mm}\\
            &
            M\in\bb{Z}^m,\ M\1\ge0,\ \1^{T}M\ge0.
        \end{aligned}
        $}
    \label{eq:PLMI_DHD}
\end{align}

To consider the dual of LMI \eqref{eq:PLMI_DHD}, we define 
\begin{align}
    \bb{Z}_{0}^{m}:= \{ X\in\bb{Z}^{m}:X_{i,i}=0 \ (i=1,\ldots,m) \}.
    \notag
\end{align}
Then, for the Lagrange dual variables 
$H\in\bb{S}_{+}^{n+m}$, $f, g\in\bb{R}_{+}^{m}$ and $X\in\bb{Z}_{0}^{m}$, 
the Lagrangian can be defined as 
\begin{align}
    \scalebox{0.85}{$
        \begin{aligned}
            &
            \cl{L}(P,M,H,f,g,X) \\
            &
            :=\trace\left(\left( \PLMIfirst + \PLMIDHDsecond \right)H\right) \\
            &
            \quad - 2f^{T}M\1 - 2\1^{T}Mg - 2\trace(MX) \\
            &
            =\trace(P\He\{AH_{11}+BH_{12}^{T}\}) \\
            &
            \quad + 2\trace(M(Y-\1f^{T}-g\1^{T}-X))
        \end{aligned}
        $}
    \notag
\end{align} 
where
\begin{align}
    \begin{aligned}
        &
        H=:\begin{bmatrix}
            H_{11} & H_{12} \\
            H_{12}^T & H_{22}
        \end{bmatrix},\ H_{11}\in\bb{S}_{+}^n,\ H_{22}\in\bb{S}_{+}^m, \\
        &
        Y:=\matrixY.
    \end{aligned}
    \notag
\end{align}
For $\cl{L}(P,M,H,f,g,X) \ge 0$ to hold 
for any $P\in\bb{S}_{++}^n$ and $M\in\bb{R}^{m \times m}$, we can select the solution
\begin{align}
    \He\{AH_{11}+BH_{12}^{T}\} \succeq 0, \ Y=\1f^{T}+g\1^{T}+X.
    \notag
\end{align}
We thus arrive at the next result.
\\
\tbf{Dual LMI (For $\Phi\in\mbf{\Phi}_{\mu,\nu}^{m}$)} \\
Find $H\in\bb{S}_{+}^{n+m}$, $f,g\in\bb{R}_{+}^m$, $X\in\bb{Z}_{0}^m$, not all zeros, such that 
\begin{align}
    \begin{aligned}
        &
        \He\{AH_{11}+BH_{12}^{T}\} \succeq 0, \\
        &\begin{multlined}
            \matrixY\\ 
            =\1f^{T}+g\1^{T}+X.
        \end{multlined}
    \end{aligned}
    \label{eq:DLMI_DHD}
\end{align}

From theorems of alternative for LMIs \cite{Scherer_EJC2006}, 
we emphasize that the primal LMI \eqref{eq:PLMI_DHD} is infeasible if and only if 
the dual LMI \eqref{eq:DLMI_DHD} is feasible.  

%%%%%%%%%%%%%%%%%%%%%%%%%%%%%%%%%%%%%%%%%%%%%%%%%%%%%%%%%%%%%%%%%%%%%%%%%%%%%
\subsection{Extraction of Destabilizing Nonlinearity by Dual LMIs}
In this section, we consider the dual LMI \eqref{eq:DLMI_DHD}. 
Due to some technical reasons, in the following subsections, 
we restrict our attention to the case 
$\mu=0$ and $\nu=1$, and we assume $\|D\|<1$.
Then, since $\|D\|<1$, we can readily see that the feedback system $\Sigma$ is well-posed 
because $\|\Phi\| \le 1$ for $\Phi\in\mbf{\Phi}_{0,1}^m$.
Moreover, since $A\in\bb{R}^{n \times n}$ is assumed to be Hurwitz stable, 
and since \eqref{eq:PLMI_DHD} requires $PA+A^TP \prec 0$, we see 
that $P\in\bb{S}_{++}^n$ is automatically satisfied if we simply require $P\in\bb{S}^n$.
Therefore, the primal LMI \eqref{eq:PLMI_DHD} and the dual LMI \eqref{eq:DLMI_DHD} 
reduce respectively to: \\
\tbf{Primal LMI (For $\Phi\in\mbf{\Phi}_{0,1}^{m}$)} \\
Find $P\in\bb{S}^n$ and $M\in\bb{R}^{m \times m}$ such that
\begin{align}\hspace{-9pt}
    \scalebox{0.95}{$
        \begin{aligned}
            &
            \PLMIfirst + \RePLMIDHDsecond \prec 0,  \\
            &
            M\in\bb{Z}^m,\ M\1 \ge 0,\ \1^{T}M \ge 0. 
        \end{aligned}
        $}
    \label{eq:RePLMI_DHD}
\end{align}
\tbf{Dual LMI (For $\Phi\in\mbf{\Phi}_{0,1}^{m}$)} \\
Find $H\in\bb{S}_{+}^{n+m}$, $f,g\in\bb{R}_{+}^m$, $X\in\bb{Z}_{0}^m$, not all zeros, 
such that 
\begin{align}
    \begin{aligned}
        &
        \He\{AH_{11}+BH_{12}^{T}\} = 0, \\
        & 
        \RematrixY=\1f^{T}+g\1^{T}+X. 
    \end{aligned}
    \label{eq:ReDLMI_DHD}
\end{align}

We note that in this dual LMI \eqref{eq:ReDLMI_DHD}
the first inequality constraint in \eqref{eq:DLMI_DHD} has been replaced by 
the equality constraint.  
Regarding this dual LMI \eqref{eq:ReDLMI_DHD}, 
we can obtain the next main result for the slope-restricted nonlinearity  case.
\begin{Theorem}\label{th:DHD}
    Suppose the dual LMI \eqref{eq:ReDLMI_DHD} 
    is feasible and has a solution $H\in\bb{S}_{+}^{n+m}$ of $\rank{H}=1$ 
    given by 
    \begin{align}
        H=\begin{bmatrix}
            h_1 \\
            h_2
        \end{bmatrix} \begin{bmatrix}
            h_1 \\
            h_2
        \end{bmatrix}^T, h_1\in\bb{R}^n, h_2\in\bb{R}^m.
        \notag
    \end{align}
    Then, the following assertions hold: 
    \begin{enumerate}\renewcommand{\labelenumi}{(\roman{enumi})}
        \item \ $h_1 \ne 0$.
        \item Let us define $z^{\ast}:=Ch_1+Dh_2$ and $w^{\ast}:=h_2$. 
        Then, there exists $\phi_{\mrm{wc}}\in\slope{0}{1}$ 
        such that $\phi_{\mrm{wc}}(z_{i}^\ast)=w_{i}^{\ast} \ (i=1,\ldots,m)$. 
        \item Let us define $\Phi_{\mrm{wc}}:=\mrm{diag}_m(\phi_{\mrm{wc}})$. 
        Then, the feedback system $\Sigma$ 
        with the nonlinearity $\Phi=\Phi_{\mrm{wc}}\in\mbf{\Phi}_{0,1}^m$ is not stable. 
        In particular, the state $x$ of the system $\Sigma$ 
        corresponding to the initial state $x(0)=h_1$ is given by $x(t)=h_1 \ (t \ge 0)$.
    \end{enumerate}
\end{Theorem}
\begin{proof}
    \underline{proof of (i)} \ 
    First, we show $H \ne 0$ to prove that $\rank{H}=1$ can happen. 
    To this end, suppose $H=0$ for contradiction. 
    Then, from the second equality constraint in \eqref{eq:ReDLMI_DHD}, 
    we see $f=0, \ g=0,$ and $X=0$, 
    which contradicts the requirements that $H,f,g,X$ are not all zeros.
    Therefore $H \ne 0$. 
    Next, we suppose $h_1=0$ for contradiction. 
    Then, $h_2 \ne 0$ since $\rank{H}=1$. 
    We see $\trace\{ h_2h_2^{T}(D^{T}-I_m) \} \ge 0$ 
    from the second equality constraint in \eqref{eq:ReDLMI_DHD}. 
    This implies $h_2^{T}Dh_2 \ge h_2^{T}h_2$, 
    which doesn't hold for $h_2 \ne 0$ since we assumed $\|D\|<1$. 
    Therefore $h_1 \ne 0$.
    
    \noindent
    \underline{proof of (ii)} \ 
    We first prove $\phi_{\mrm{wc}}:\bb{R}\to\bb{R}$ such that
    $\phi_{\mrm{wc}}(z_{i}^\ast)=w_{i}^{\ast} \ (i=1,\ldots,m)$ is well-defined. 
    To this end, it suffices to prove that 
    if $z_{i}^{\ast}=z_{j}^{\ast}$ then $w_{i}^{\ast}=w_{j}^{\ast} \ (i \ne j)$. 
    To prove this assertion, we note 
    that the second equality constraint in \eqref{eq:ReDLMI_DHD} 
    can be rewritten equivalently as 
    \begin{align}
        w^{\ast}(z^{\ast}-w^{\ast})^{T}=\1f^{T}+g\1^{T}+X.
        \label{eq:proof_DHD_well-defined}
    \end{align}
    Then, if $z_i^{\ast}=z_j^{\ast} \ (i \ne j)$,  
    we have from \eqref{eq:proof_DHD_well-defined} that  
    \begin{align}
        \scalebox{0.95}{$
            \begin{aligned}
                -(w_i^\ast-w_j^\ast)^2
                &
                =(e_i-e_j)^{T}w^{\ast}(z^{\ast}-w^{\ast})^{T}(e_i-e_j) \\
                &
                =(e_i-e_j)^{T}(\1f^{T}+g\1^{T}+X)(e_i-e_j) \ge 0.
            \end{aligned}
            $}
        \notag
    \end{align}
    where $e_i \ (i=1,\ldots,m)$ are the standard basis. 
    This clearly shows $w_{i}^{\ast}=w_{j}^{\ast} \ (i \ne j)$. 
    
    We finally prove the existence of $\phi_{\mrm{wc}}\in\slope{0}{1}$ such that  
    $\phi_{\mrm{wc}}(z_{i}^\ast)=w_{i}^{\ast} \ (i=1,\ldots,m)$. 
    To this end, we need to prove that if 
    $z_{i}^{\ast}=0$ then $w_{i}^{\ast}=0$ from Definition~\ref{de:slope-restricted}.   
    This can be verified from \eqref{eq:proof_DHD_well-defined} 
    since if $z_{i}^{\ast}=0$ then we have 
    \begin{multline}
        -w_i^{*2}=e_i^{T}w^{\ast}(z^{\ast}-w^{\ast})^{T}e_i\\
        =e_i^{T}(\1f^{T}+g\1^{T}+X)e_i \ge 0.  
        \notag
    \end{multline}
    This clearly shows $w_i^{\ast}=0$.  In addition, again from 
    \eqref{eq:proof_DHD_well-defined}, we have 
    \begin{align}
        \begin{aligned}
            &
            (w_i^{\ast}-w_j^{\ast})((z_i^{\ast}-z_j^{\ast})-(w_i^{\ast}-w_j^{\ast})) \\
            &
            =(e_i-e_j)^{T}w^{\ast}(z^{\ast}-w^{\ast})^{T}(e_i-e_j) \\
            &
            =(e_i-e_j)^{T}(\1f^{T}+g\1^{T}+X)(e_i-e_j) \ge 0.
        \end{aligned}
        \notag
    \end{align}
    Therefore, if $z_i^{\ast} \ne z_j^{\ast} \ (i \ne j)$, we have 
    \begin{align}
        \frac{w_i^{\ast}-w_j^{\ast}}{z_i^{\ast}-z_j^{\ast}} 
        \left(1-\frac{w_i^{\ast}-w_j^{\ast}}{z_i^{\ast}- z_j^{\ast}} \right) \ge 0.
        \label{eq:proof_DHD_slope_1}
    \end{align}
    %
    %Moreover, if $z_i^{\ast} \ne 0 \ (i=1,\ldots,m)$, we have
    %%
    %\begin{align}
    %    \frac{w_i^{\ast}}{z_i^{\ast}} \left( 1-\frac{w_i^{\ast}}{z_i^{\ast}} \right) \ge 0.
    %    \label{eq:proof_DHD_slope_2}
    %\end{align}
    %%
    From \eqref{eq:proof_DHD_well-defined}, we also have 
    \begin{align*}
        \begin{aligned}
            &
            w_i^{\ast}(z_i^{\ast}-w_i^{\ast}) \\
            &
            =e_i^{T}w^{\ast}(z^{\ast}-w^{\ast})^{T}e_i \\
            &
            =e_i^{T}(\1f^{T}+g\1^{T}+X)e_i \ge 0.
        \end{aligned}
    \end{align*}
    Therefore, if $z_i^{\ast} \ne 0 \ (i=1,\ldots,m)$, we have 
    \begin{align}
        \frac{w_i^{\ast}}{z_i^{\ast}} \left( 1-\frac{w_i^{\ast}}{z_i^{\ast}} \right) \ge 0.
        \label{eq:proof_DHD_slope_2}
    \end{align}
    We then conclude that 
    \eqref{eq:proof_DHD_slope_1} and \eqref{eq:proof_DHD_slope_2} 
     clearly show the existence of  $\phi_\mrm{wc}\in\slope{0}{1}$ such that 
    $\phi_\mrm{wc}(z^\ast_i)=w^\ast_i\ (i=1,\cdots,m)$.  
        
    \noindent
    \underline{proof of (iii)} \ 
    It suffices to prove 
    that $x(t)=h_1$, $z(t)=z^\ast$, $w(t)=w^\ast$ $(t\ge 0)$ satisfy \eqref{eq:systemG} and  \eqref{eq:nonlinear} 
    in the case $\Phi=\Phi_{\mrm{wc}}$ and $x(0)=h_1$.
    From the first equality constraint in \eqref{eq:ReDLMI_DHD}, we have 
    \begin{align}
        \He\{ Ah_1h_1^{T}+Bh_2h_1^{T} \} = \He\{ (Ah_1+Bh_2)h_1^{T} \} = 0.  
        \notag
    \end{align}
    Since $h_1 \ne 0$ as proved, we see $Ah_1+Bh_2=0$ from \cite{green}. 
    Therefore, we can obtain 
    \begin{align}
        \begin{aligned}
            &
            \frac{d}{dt}h_1=0=Ah_1+Bh_2=Ah_1+Bw^\ast, \\
            &
            z(t)=z^\ast=Ch_1+Dh_2=Ch_1+Dw^\ast, \\
            &
            w(t)=w^\ast=\Phi_\mrm{wc}(z^\ast).
        \end{aligned}
        \notag
    \end{align}
    This completes the proof.
\end{proof}

%%%%%%%%%%%%%%%%%%%%%%%%%%%%%%%%%%%%%%%%%%%%%%%%%%%%%%%%%%%%%%%%%%%%%%%%%%%%%
\subsection{Concrete Construction of Destabilizing Nonlinearity}
From Theorem \ref{th:DHD}, we see that 
any $\phi\in\slope{0}{1}$ with $\phi(z_i^{\ast})=w_i^{\ast} \ (i=1,\ldots,m)$ is 
a destabilizing nonlinearity. 
One of such destabilizing nonlinear (piecewise linear) operators
can be constructed by following the next procedure:
\begin{enumerate}
    \item Define the set $\mathcal{Z}_0 := \{0, z_1^{\ast}, z_2^{\ast}, \ldots, z_m^{\ast}\}$ 
    (by leaving only one if they have duplicates)
    and compose the series $\bar{z}_1,\bar{z}_2, \ldots, \bar{z}_{l}$, $l=\lvert \mathcal{Z}_0\rvert$, 
    such that each $\bar{z}_i\ (1 \le i \le l)$ is the $i$-th smallest value of $\mathcal{Z}_0$. 
    Similarly, define the series $\bar{w}_1,\bar{w}_2,\ldots,\bar{w}_{l}$ for the set $\{0, w_1^{\ast}, w_2^{\ast}, \ldots, w_m^{\ast} \}$. 
    \item 
    Define $\phi_{\mrm{wc}}$ as follows:
    \begin{align}\hspace{-10pt}
        &\phi_{\mrm{wc}}(z)\notag\\
        &{}= 
        \begin{cases}
            \bar{w}_1, &z < \bar{z}_1, \\\displaystyle
            \frac{\bar{w}_{i+1}-\bar{w}_i}{\bar{z}_{i+1}-\bar{z}_i}(z-\bar{z}_i)+w_i, 
            &\parbox[t]{.4\linewidth}{
                $\bar{z}_i \le z < \bar{z}_{i+1}$\\
                $(i = 1, \ldots, l-1)$,
            }\\ 
            \bar{w}_{l}, &\bar{z}_{l} \le z. 
        \end{cases}
        \notag
    \end{align}%
\end{enumerate}
%%%%%%%%%%%%%%%%%%%%%%%%%%%%%%%%%%%%%%%%%%%%%%%%%%%%%%%%%%%%%%%%%%%%%%%%%%%%%
\subsection{Numerical Examples}
We demonstrate the soundness of Theorem \ref{th:DHD}.
Let us consider the case where the coefficient matrices in \eqref{eq:systemG} are given by
\begin{align}
    \scalebox{0.9}{$
        \begin{aligned}
            A=\begin{bmatrix*}[r]
                -2.11 &  0.94 \\
                0.77 & -0.46 \\
            \end{bmatrix*},\ 
            &B=\begin{bmatrix*}[r]
                0.28 & -0.85 & -0.94 & -0.71 \\
                -0.82 & -0.64 &  0.45 &  0.27 \\
            \end{bmatrix*}, \\
            C=\begin{bmatrix*}[r]
                0.58 & -0.39 \\
                0.13 & -0.36 \\
                -0.25 &  0.57 \\
                0.64 &  0.01 \\
            \end{bmatrix*},\ 
            &D=\begin{bmatrix*}[r]
                -0.24 & -0.28 &  0.32 & -0.04 \\
                0.23 & -0.42 &  0.24 &  0.29 \\
                -0.34 & -0.43 &  0.26 & -0.08 \\
                0.42 &  0.27 &  0.46 &  0.45 \\
            \end{bmatrix*}.
        \end{aligned}
        $}
    \notag
\end{align}
For this system, the dual LMI \eqref{eq:ReDLMI_DHD} turns out to be feasible, 
and the resulting dual solution $H$ is numerically verified to be $\rank{H}=1$. 
The full-rank factorization of $H$ as well as $z^{\ast},w^{\ast}\in\bb{R}^4$ 
in Theorem \ref{th:DHD} are obtained as
\begin{align}
    \scalebox{0.7}{$
        h_1=
        \begin{bmatrix*}[r]
            0.7321\\
            1.5697\\
        \end{bmatrix*},\ 
        h_2=
        \begin{bmatrix*}[r]
            -0.0907\\
            -0.1234\\
            0.0111\\
            0.0000\\
        \end{bmatrix*},\ 
        z^\ast=
        \begin{bmatrix*}[r]
            -0.1277\\
            -0.4363\\
            0.7985\\
            0.4179\\
        \end{bmatrix*},\ 
        w^\ast=
        \begin{bmatrix*}[r]
            -0.0907\\
            -0.1234\\
            0.0111\\
            0.0000\\
        \end{bmatrix*}(=h_2).  
        $}
    \notag
\end{align}
It is clear that $h_1 \ne 0$ (the assertion (i) of Theorem \ref{th:DHD}). 
Fig. \ref{fig:phi_wc_DHD} shows the input-output map 
of the destabilizing nonlinear (piecewise linear) operator $\phi_{\mrm{wc}}$ 
constructed by following the procedure in the preceding subsection. 
From Fig. \ref{fig:phi_wc_DHD}, 
we can readily see $\phi_{\mrm{wc}}\in\slope{0}{1}$ (the assertion (ii)). 
Fig. \ref{fig:state_trajectory_DHD} shows the vector field 
of the system $\Sigma$ with the nonlinearity $\Phi=\Phi_{\mrm{wc}}$, 
together with the state trajectories from initial states $x(0)=h_1$ and $x(0)=[-0.5 \ -0.5]^T$.
The state trajectory from the initial state $x(0)=[-0.5 \ -0.5]^T$ converges to the origin. 
However, as proved in Theorem \ref{th:DHD}, 
the state trajectory from the initial state $x(0)=h_1$ does not evolve and 
$x(0)=h_1$ is confirmed to be an equilibrium point of the system $\Sigma$ (the assertion (iii)).
\begin{figure}[tbp]
    \centering
    \vspace*{-5mm}
    \includegraphics[scale=0.6]{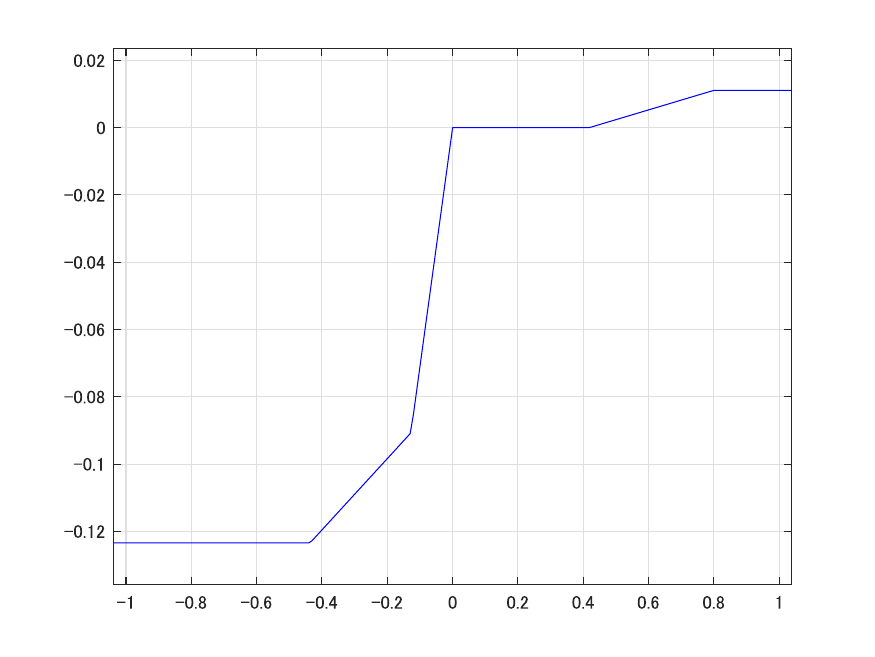}
    \vspace*{-10mm}
    \caption{The input-output map of the detected $\phi_{\mrm{wc}}$.}
    \label{fig:phi_wc_DHD}
    \centering
    \includegraphics[scale=0.6]{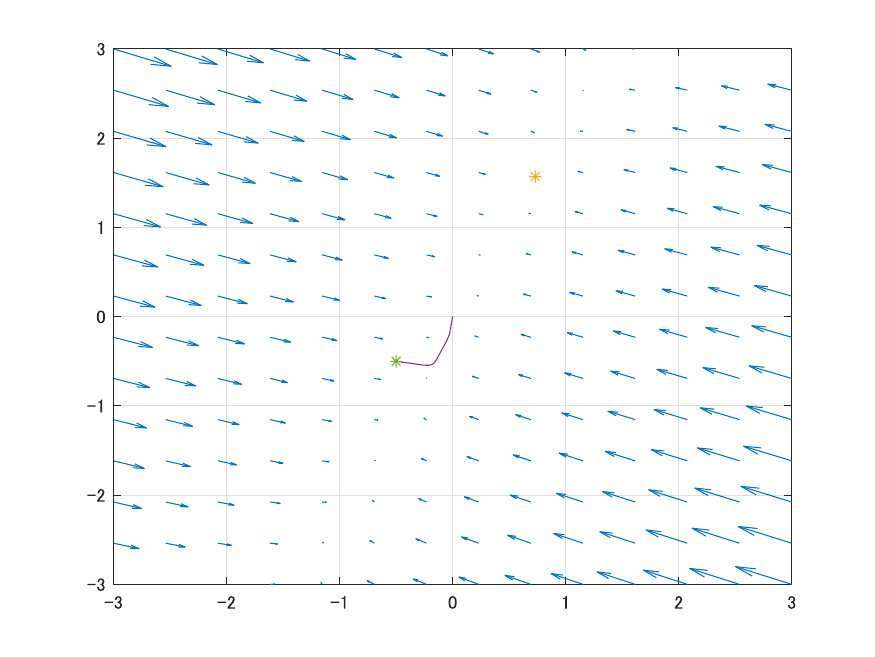}
    \vspace*{-10mm}
    \caption{The state trajectories with initial conditions $x(0)=h_1$ and $x(0)=[-0.5\ -0.5]^T$.}
    \label{fig:state_trajectory_DHD}
    \vspace*{-4mm}
\end{figure}
%

%%%%%%%%%%%%%%%%%%%%%%%%%%%%%%%%%%%%%%%%%%%%%%%%%%%%%%%%%%%%%%%%%%%%%%%%%%%%%%%%%%%%%%%%%%%%%%%%%%%%%%%%%%%%%%%%%%%%%%%%%%%%%%%%%%%%%%%%%%%%%%%%%%%%%%%%%%%%%
\section{Main Results for Slope-Restricted and Odd Nonlinearity
$\Phi\in\mbf{\Phi}_{\mu,\nu}^{m}\cap\mbf{\Phi}_{\mrm{odd}}^{m}$}
%%%%%%%%%%%%%%%%%%%%%%%%%%%%%%%%%%%%%%%%%%%%%%%%%%%%%%%%%%%%%%%%%%%%%%%%%%%%%
\subsection{LMIs Ensuring Absolute Stability and Their Dual}
We can readily obtain the next lemma 
from Proposition \ref{pr:IQC_basic} and Lemma \ref{lemma:OZF_basic}. 
\begin{Lemma}\label{lemma:DD}
    Let $\mu \le 0 \le \nu$. 
    Then, the system $\Sigma$ consisting of \eqref{eq:systemG} and \eqref{eq:nonlinear} 
    is absolutely stable for $\Phi\in\mbf{\Phi}_{\mu,\nu}^{m}\cap\mbf{\Phi}_{\mrm{odd}}^m$ 
    if there exist $P\in\bb{S}_{++}^n$ and $M\in\bb{DD}^m$ such that 
    \begin{align}
    \scalebox{0.95}{$
     \begin{aligned}
        \PLMIfirst + \PLMIDHDsecond \prec 0.  
        \label{eq:lem_DD}
     \end{aligned}$}
    \end{align}
\end{Lemma}

For the same reason as the slope-restricted case, 
we derived the dual of LMI \eqref{eq:lem_DD}. 
We first note that LMI \eqref{eq:lem_DD} can be rewritten, equivalently, as follows: \\
\tbf{Primal LMI (For $\Phi\in\mbf{\Phi}_{\mu,\nu}^{m}\cap\mbf{\Phi}_{\mrm{odd}}^{m}$)} \\
Find $P\in\bb{S}_{++}^n$, $M_{\mrm{d}}\in\bb{D}^m$, 
and $M_{\mrm{od}},\overline{M}_{\mrm{od}}\in\bb{OD}^m$ such that 
\begin{align}
    \scalebox{0.85}{$
        \begin{aligned}
            &
            \PLMIfirst + \PLMIDDsecond \prec 0, \hspace*{-10mm}\\
            &
            (M_{\mrm{d}}-\overline{M}_{\mrm{od}})\1\ge0, \ \1^{T}(M_{\mrm{d}}-\overline{M}_{\mrm{od}})\ge0, 
            \\ &
            \overline{M}_{\mrm{od}}-M_{\mrm{od}}\ge0, \ \overline{M}_{\mrm{od}}+M_{\mrm{od}}\ge0.   
        \end{aligned}
        $}
    \label{eq:PLMI_DD}
\end{align}

Then, for the Lagrange dual variables 
$H\in\bb{S}_{+}^{n+m}$, $f,g\in\bb{R}_{+}^m$, and $X,Z\in\bb{Z}_{0}^m$, 
the Lagrangian can be defined as 
\begin{align}
    \scalebox{0.76}{$
        \begin{aligned}
            &
            \cl{L}(P,M_{\mrm{d}},M_{\mrm{od}},\overline{M}_{\mrm{od}},H,f,g,X,Z) \\
            &
            :=
            \trace\left(\left( \PLMIfirst + \PLMIDDsecond \right)H\right)  \\
            & \quad
            -2f^{T}(M_{\mrm{d}}-\overline{M}_{\mrm{od}})\1 
             -2\1^{T}(M_{\mrm{d}}-\overline{M}_{\mrm{od}})g 
            \\ & \quad
            +2\trace((\overline{M}_{\mrm{od}}-M_{\mrm{od}})X)
            +2\trace((M_{\mrm{od}}+\overline{M}_{\mrm{od}})Z) \\
            &
            =\trace\{P\He\{AH_{11}+BH_{12}^T\}\} \\
            & \quad
            +2\trace(M_{\mrm{d}}(Y-\1f^{T}-g\1^{T}))\\
            & \quad
            +2\trace(M_{\mrm{od}}(Y-X+Z))\\
            & \quad
            +2\trace(\overline{M}_{\mrm{od}}(\1f^{T}+g\1^{T}+X+Z))
        \end{aligned}
        $}
    \notag
\end{align}
where
\begin{align}
    \begin{aligned}
        &
        H=:\begin{bmatrix}
            H_{11} & H_{12} \\
            H_{12}^T & H_{22}
        \end{bmatrix},\ H_{11}\in\bb{S}_{+}^n,\ H_{22}\in\bb{S}_{+}^m, \\
        &
        Y:=\matrixY.
    \end{aligned}
    \notag
\end{align}
For $\cl{L}(P,M_{\mrm{d}},M_{\mrm{od}},\overline{M}_{\mrm{od}},H,f,g,X,Z) \ge 0$ 
to hold for any $P\in\bb{S}_{++}^n$, $M_{\mrm{d}}\in\bb{D}^m$, 
and $M_{\mrm{od}},\overline{M}_{\mrm{od}}\in\bb{OD}^m$, we see 
\begin{align}
    \begin{aligned}
        &
        \He\{AH_{11}+BH_{12}^T\} \succeq 0, \\
        &
        \projection{d}{Y}=\projection{d}{\1f^{T}+g\1^T}, \\
        &
        \projection{od}{Y}=\projection{od}{X-Z}, \\
        &
        \projection{od}{X+Z}=-\projection{od}{\1f^{T}+g\1^{T}}.
    \end{aligned}
    \notag
\end{align}
We thus arrive at the next result. \\
\tbf{Dual LMI (For $\Phi\in\mbf{\Phi}_{\mu,\nu}^{m}\cap\mbf{\Phi}_{\mrm{odd}}^{m}$)} \\
Find $H\in\bb{S}_{+}^{n+m}$, $f,g\in\bb{R}_{+}^m$, $X,Z\in\bb{Z}_{0}^m$, not all zeros, 
such that 
\begin{align}
    \scalebox{0.9}{$
        \begin{aligned}
            &
            \He\{AH_{11}+BH_{12}^T\} \succeq 0, \\
            &
            \begin{multlined}
                \projection{d}{\matrixY}\\
                =\projection{d}{\1f^{T}+g\1^T}, 
            \end{multlined}\\
            &
            \begin{multlined}
                \projection{od}{\matrixY}\\
                =\projection{od}{X-Z}, 
            \end{multlined}\\
            & 
            \projection{od}{X+Z}=-\projection{od}{\1f^{T}+g\1^{T}}. 
        \end{aligned}
        $}
    \label{eq:DLMI_DD}
\end{align}
%

%%%%%%%%%%%%%%%%%%%%%%%%%%%%%%%%%%%%%%%%%%%%%%%%%%%%%%%%%%%%%%%%%%%%%%%%%%%%%
\subsection{Extraction of Destabilizing Nonlinearity by Dual LMIs}
Suppose $\mu=0$, $\nu=1$, and $\|D\|<1$ due to the same reasons 
as the slope-restricted case. 
Then, the primal LMI \eqref{eq:PLMI_DD} and the dual LMI \eqref{eq:DLMI_DD} 
reduce respectively to: \\
\tbf{Primal LMI (For $\Phi\in\mbf{\Phi}_{0,1}^{m}\cap\mbf{\Phi}_{\mrm{odd}}^{m}$)} \\
Find $P\in\bb{S}^n$, $M_{\mrm{d}}\in\bb{D}^m$, 
and $M_{\mrm{od}},\overline{M}_{\mrm{od}}\in\bb{OD}^m$ such that 
\begin{align}
    \scalebox{0.85}{$
        \begin{aligned}
            &
            \PLMIfirst + \RePLMIDDsecond \prec 0, \hspace*{-10mm}\\
            &
            (M_{\mrm{d}}-\overline{M}_{\mrm{od}})\1\ge0, \ \1^{T}(M_{\mrm{d}}-\overline{M}_{\mrm{od}})\ge0, 
            \\ &
            \overline{M}_{\mrm{od}}-M_{\mrm{od}}\ge0, \ \overline{M}_{\mrm{od}}+M_{\mrm{od}}\ge0.
        \end{aligned}
        $}
    \label{eq:RePLMI_DD}
\end{align}
\\
\tbf{Dual LMI (For $\Phi\in\mbf{\Phi}_{0,1}^{m}\cap\mbf{\Phi}_{\mrm{odd}}^{m}$)}\\
Find $H\in\bb{S}_{+}^{n+m}$, $f,g\in\bb{R}_{+}^m$, $X,Z\in\bb{Z}_{0}^m$, not all zeros, 
such that 
\begin{align}
    \scalebox{0.9}{$
        \begin{aligned}
            &
            \He\{AH_{11}+BH_{12}^T\} = 0, \\
            &
            \begin{multlined}
                \projection{d}{\RematrixY}
                =\projection{d}{\1f^{T}+g\1^T}, 
            \end{multlined}\\
            &
            \begin{multlined}
                \projection{od}{\RematrixY}
                =\projection{od}{X-Z}, 
            \end{multlined}\\
            &
            \projection{od}{X+Z}=-\projection{od}{\1f^{T}+g\1^{T}}. 
        \end{aligned}
        $} \
    \label{eq:ReDLMI_DD}
\end{align}

Regarding this dual LMI \eqref{eq:ReDLMI_DD}, 
we can obtain the next main result for the slope-restricted and odd nonlinearity case.
\begin{Theorem}\label{th:DD}
    Suppose the dual LMI \eqref{eq:ReDLMI_DD} is feasible 
    and has a solution $H\in\bb{S}_{+}^{n+m}$ of $\rank{H}=1$ given by 
    \begin{align}
        H=\begin{bmatrix}
            h_1 \\
            h_2
        \end{bmatrix} \begin{bmatrix}
            h_1 \\
            h_2
        \end{bmatrix}^T, h_1\in\bb{R}^n, h_2\in\bb{R}^m.
        \notag
    \end{align}
    Then, the following assertions hold: 
    \begin{enumerate}\renewcommand{\labelenumi}{(\roman{enumi})}
        \item $h_1 \ne 0$. 
        \item Let us define $z^{\ast}:=Ch_1+Dh_2$ and $w^{\ast}:=h_2$. 
        Then, there exists an odd nonlinearity $\phi_{\mrm{wc}}\in\slope{0}{1}$ 
        such that 
        $\phi_{\mrm{wc}}(z_{i}^\ast)=w_{i}^{\ast}$ and 
        $\phi_{\mrm{wc}}(-z_{i}^\ast)=-w_{i}^{\ast} \ (i=1,\ldots,m)$.
        \item Let us define $\Phi_{\mrm{wc}}:=\mrm{diag}_m(\phi_{\mrm{wc}})$. 
        Then, the feedback system $\Sigma$ with
        the nonlinearity $\Phi=\Phi_{\mrm{wc}}\in\mbf{\Phi}_{0,1}^m\cap\mbf{\Phi}_{\mrm{odd}}^m$ 
        is not stable.
        In particular, the state $x$ of the system $\Sigma$ 
        corresponding to the initial state $x(0)=h_1$ is given by $x(t)=h_1 \ (t \ge 0)$.
    \end{enumerate}
\end{Theorem}
\begin{proof}
    \underline{proof of (i)} \ 
    First, we prove $H \ne 0$ to prove that $\rank{H}=1$ can be happen. 
    To this end, suppose $H=0$ for contradiction. 
    Then, from the second equality constraint in \eqref{eq:ReDLMI_DD}, 
    we see $f=0$ and $g=0$ 
    and from the fourth equality constraint in \eqref{eq:ReDLMI_DD}, 
    we see $X=0$ and $Z=0$,  
    which contradicts the requirements that $H,f,g,X,Z$ are not all zeros.
    Therefore $H \ne 0$. 
    Next, we suppose $h_1=0$ for contradiction. 
    Then, $h_2 \ne 0$ since $\rank{H}=1$.  
    We see $\trace\{ h_2h_2^{T}(D^{T}-I_m) \} \ge 0$ 
    from the second equality constraint in \eqref{eq:ReDLMI_DD}. 
    This implies $h_2^{T}Dh_2 \ge h_2^{T}h_2$, 
    which doesn't hold for $h_2 \ne 0$ since we assumed $\|D\|<1$. 
    Therefore $h_1 \ne 0$. 
    
    \noindent
    \underline{proof of (ii)} \ 
    We first prove that odd $\phi_{\mrm{wc}}:\bb{R}\to\bb{R}$ 
    such that 
    $\phi_{\mrm{wc}}(z_{i}^\ast)=w_{i}^{\ast}$ and 
    $\phi_{\mrm{wc}}(-z_{i}^\ast)=-w_{i}^{\ast} \ (i=1,\ldots,m)$
    is well-defined. 
    To this end, it suffices to prove that 
    (iia) if $z_{i}^{\ast}=0$ then $w_{i}^{\ast}=0$, 
    (iib) if $z_{i}^{\ast}=z_{j}^{\ast}$ then $w_{i}^{\ast}=w_{j}^{\ast} \ (i \ne j)$, 
    and (iic) if $z_{i}^{\ast}=-z_{j}^{\ast}$ then $w_{i}^{\ast}=-w_{j}^{\ast} \ (i \ne j)$. 
    To prove these assertions, we note 
    that the second and third equality constraints in \eqref{eq:ReDLMI_DD} 
    can be rewritten, equivalently and respectively, as 
    \begin{align}
        &
        \projection{d}{w^{\ast}(z^{\ast}-w^{\ast})^T}=\projection{d}{\1f^{T}+g\1^T},
        \label{eq:proof_DD_well-defined_1} \\
        &
        \projection{od}{w^{\ast}(z^{\ast}-w^{\ast})^T}=\projection{od}{X-Z}.
        \label{eq:proof_DD_well-defined_2}
    \end{align}

    To prove (iia), suppose $z_{i}^{\ast}=0$.  
    Then, we have from \eqref{eq:proof_DD_well-defined_1} that
    \begin{align}
        -w_i^{*2}=e_i^{T}w^{\ast}(z^{\ast}-w^{\ast})^{T}e_i=e_i^{T}(\1f^{T}+g\1^{T})e_i \ge 0.  
        \notag
    \end{align}
    This clearly shows $w_i^{\ast}=0$.
    To prove (iib), suppose $z_i^{\ast}=z_j^{\ast} \ (i \ne j)$. 
    Then, from \eqref{eq:proof_DD_well-defined_1}, \eqref{eq:proof_DD_well-defined_2}, 
     and the fourth equality constraint in \eqref{eq:ReDLMI_DD}, we have 
    \begin{align}
        \begin{aligned}
            &
            -(w_i^{\ast}-w_j^{\ast})^2 \\
            &
            =(e_i-e_j)^{T}w^{\ast}(z^{\ast}-w^{\ast})^{T}(e_i-e_j) \\
            &
            =e_i^{T}(\1f^{T}+g\1^{T})e_i + e_j^{T}(\1f^{T}+g\1^T)e_j \\
            & \quad
            -e_j^{T}(X-Z)e_i - e_i^{T}(X-Z)e_j \\
            &
            \ge e_i^{T}(\1f^{T}+g\1^T)e_i + e_j^{T}(\1f^{T}+g\1^T)e_j \\
            & \quad
            +e_j^{T}(X+Z)e_i + e_i^{T}(X+Z)e_j \\
            &
            =e_i^{T}(\1f^{T}+g\1^T)e_i + e_j^{T}(\1f^{T}+g\1^T)e_j \\
            & \quad
            -e_j^{T}(\1f^{T}+g\1^T)e_i - e_i^{T}(\1f^{T}+g\1^T)e_j \\
            &
            =f_i+g_i+f_j+g_j-(f_i+g_j)-(f_j+g_i)=0.
        \end{aligned}
        \label{eq:proof_DD_well-defined_3}
    \end{align}
    This clearly shows $w_{i}^{\ast}=w_{j}^{\ast}$. 
    To prove (iic), suppose $z_i^{\ast}=-z_j^{\ast} \ (i \ne j)$. 
    Then, similarly to the proof of (iib), we have 
    \begin{align}
        \begin{aligned}
            &
            -(w_i^{\ast}+w_j^{\ast})^2 \\
            &
            =(e_i+e_j)^{T}w^{\ast}(z^{\ast}-w^{\ast})^{T}(e_i+e_j) \\
            &
            =e_i^{T}(\1f^{T}+g\1^{T})e_i + e_j^{T}(\1f^{T}+g\1^T)e_j \\
            & \quad
            +e_j^{T}(X-Z)e_i + e_i^{T}(X-Z)e_j \\
            &
            \ge e_i^{T}(\1f^{T}+g\1^T)e_i + e_j^{T}(\1f^{T}+g\1^T)e_j \\
            & \quad
            +e_j^{T}(X+Z)e_i + e_i^{T}(X+Z)e_j = 0.
        \end{aligned}
        \label{eq:proof_DD_well-defined_4}
    \end{align}
    This clearly shows $w_{i}^{\ast}=-w_{j}^{\ast} \ (i \ne j)$.
    
    We finally prove the existence of odd $\phi_{\mrm{wc}}\in\slope{0}{1}$ such that
    $\phi_{\mrm{wc}}(z_{i}^\ast)=w_{i}^{\ast}$ and 
    $\phi_{\mrm{wc}}(-z_{i}^\ast)=-w_{i}^{\ast} \ (i=1,\ldots,m)$.   
    From \eqref{eq:proof_DD_well-defined_3}, we have 
    \begin{align}
        \begin{aligned}
            &
            (w_i^{\ast}-w_j^{\ast})((z_i^{\ast}-z_j^{\ast})-(w_i^{\ast}-w_j^{\ast})) \\
            &
            =(e_i-e_j)^{T}w^{\ast}(z^{\ast}-w^{\ast})^{T}(e_i-e_j) \ge 0. 
        \end{aligned}
        \notag
    \end{align}
    Therefore, if $z_i^{\ast} \ne z_j^{\ast} \ (i \ne j)$, we have 
    \begin{align}
        \frac{w_i^{\ast}-w_j^{\ast}}{z_i^{\ast}-z_j^{\ast}} 
        \left(1-\frac{w_i^{\ast}-w_j^{\ast}}{z_i^{\ast}-z_j^{\ast}} \right) \ge 0.
        \label{eq:proof_DD_slope_1}
    \end{align}
    From \eqref{eq:proof_DD_well-defined_4}, we also have 
    \begin{align}
        \begin{aligned}
            &
            (w_i^{\ast}+w_j^{\ast})((z_i^{\ast}+z_j^{\ast})-(w_i^{\ast}+w_j^{\ast})) \\
            &
            =(e_i+e_j)^{T}w^{\ast}(z^{\ast}-w^{\ast})^{T}(e_i+e_j) \ge 0. 
        \end{aligned}
        \notag
    \end{align}
    Therefore, if $z_i^{\ast} \ne -z_j^{\ast} \ (i \ne j)$, we have 
    \begin{align}
        \frac{w_i^{\ast}-(-w_j^{\ast})}{z_i^{\ast}-(-z_j^{\ast})} 
        \left(1-\frac{w_i^{\ast}-(-w_j^{\ast})}{z_i^{\ast}-(-z_j^{\ast})} \right) \ge 0.
        \label{eq:proof_DD_slope_2}
    \end{align}
    Moreover, from \eqref{eq:proof_DD_well-defined_1}, we have 
    \begin{align}
    \scalebox{0.9}{$
    \begin{aligned}
        w_i^{\ast}(z_i^{\ast}-w_i^{\ast})&=e_i^{T}w^{\ast}(z^{\ast}-w^{\ast})^{T}e_i
        =e_i^{T}(\1f^{T}+g\1^T)e_i 
        \ge 0.
    \end{aligned}$}
        \notag
    \end{align}
    Therefore, if $z_i^{\ast} \ne 0 \ (i=1,\ldots,m)$, we have
    \begin{align}
        \frac{w_i^{\ast}}{z_i^{\ast}} \left( 1-\frac{w_i^{\ast}}{z_i^{\ast}} \right) \ge 0. 
        \label{eq:proof_DD_slope_3}
    \end{align}
    We then conclude that 
    \eqref{eq:proof_DD_slope_1}, \eqref{eq:proof_DD_slope_2}, 
    and \eqref{eq:proof_DD_slope_3}
    clearly show the existence of  odd $\phi_\mrm{wc}\in\slope{0}{1}$ such that 
    $\phi_{\mrm{wc}}(z_{i}^\ast)=w_{i}^{\ast}$ and 
    $\phi_{\mrm{wc}}(-z_{i}^\ast)=-w_{i}^{\ast} \ (i=1,\ldots,m)$.   
    
    \noindent
    \underline{proof of (iii)} \ 
    Omitted since the proof is exactly the same as the proof of (iii) in Theorem \ref{th:DHD}.
\end{proof}
%%
%\addtolength{\textheight}{-6cm}

%%%%%%%%%%%%%%%%%%%%%%%%%%%%%%%%%%%%%%%%%%%%%%%%%%%%%%%%%%%%%%%%%%%%%%%%%%%%%
\subsection{Concrete Construction of Destabilizing Nonlinearity}
From Theorem \ref{th:DD}, we see that 
any odd $\phi\in\slope{0}{1}$ 
with $\phi(z_i^{\ast})=w_i^{\ast}, \phi(-z_i^{\ast})=-w_i^{\ast} \ (i=1,\ldots,m)$ is 
a destabilizing nonlinearity. 
One of such destabilizing nonlinear (piecewise linear) operator
can be constructed by following the next procedure:
\begin{enumerate}
    \item Define the set 
    \begin{align*}
        \mathcal{Z}_0 := \{0, z_1^{\ast}, -z_1^{\ast}, z_2^{\ast}, -z_2^{\ast}, \ldots, z_m^{\ast}, -z_m^{\ast}\}
    \end{align*}
     (by leaving only one if they have duplicates) 
    and compose the series $\bar{z}_1,\bar{z}_2, \ldots, \bar{z}_{l}$, $l=\lvert \mathcal{Z}_0\rvert$, 
    such that each $\bar{z}_i\ (1 \le i \le l)$ is the $i$-th smallest value of $\mathcal{Z}_0$. 
    Similarly, define the series $\bar{w}_1,\bar{w}_2,\ldots,\bar{w}_{l}$ for the set $\{0, w_1^{\ast}, -w_1^{\ast}, w_2^{\ast}, -w_2^{\ast}, \ldots, w_m^{\ast}, -w_m^{\ast} \}$. 
    \item 
    Define $\phi_{\mrm{wc}}$ as follows:
    \begin{align}\hspace{-10pt}
        &\phi_{\mrm{wc}}(z)\notag\\
        &{}= 
        \begin{cases}
            \bar{w}_1, &z \le \bar{z}_1, \\\displaystyle
            \frac{\bar{w}_{i+1}-\bar{w}_i}{\bar{z}_{i+1}-\bar{z}_i}(z-\bar{z}_i)+w_i, 
            &\parbox[t]{.4\linewidth}{
                $\bar{z}_i \le z \le \bar{z}_{i+1}$\\
                $(i = 1, \ldots, l-1)$,
            }\\ 
            \bar{w}_{l}, &\bar{z}_{l} \le z.
        \end{cases}
        \notag
    \end{align}
\end{enumerate}

%%%%%%%%%%%%%%%%%%%%%%%%%%%%%%%%%%%%%%%%%%%%%%%%%%%%%%%%%%%%%%%%%%%%%%%%%%%%%
\subsection{Numerical Examples}
\begin{figure}[tbp]
    \centering
    \vspace*{-5mm}
    \includegraphics[scale=0.6]{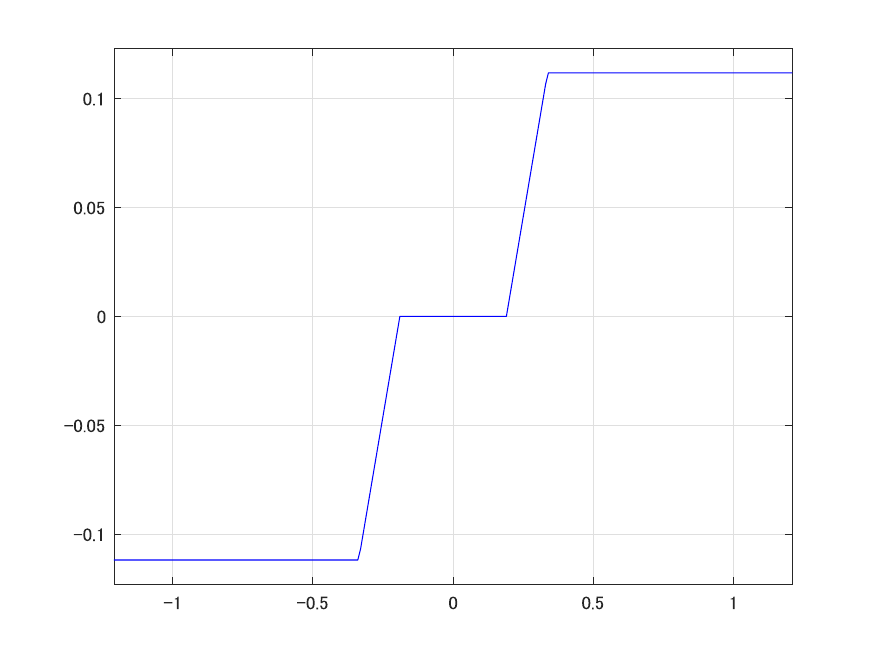}
    \vspace*{-10mm}
    \caption{The input-output map of the detected $\phi_{\mrm{wc}}$.}
    \label{fig:phi_wc_DD}
    \centering
    \includegraphics[scale=0.6]{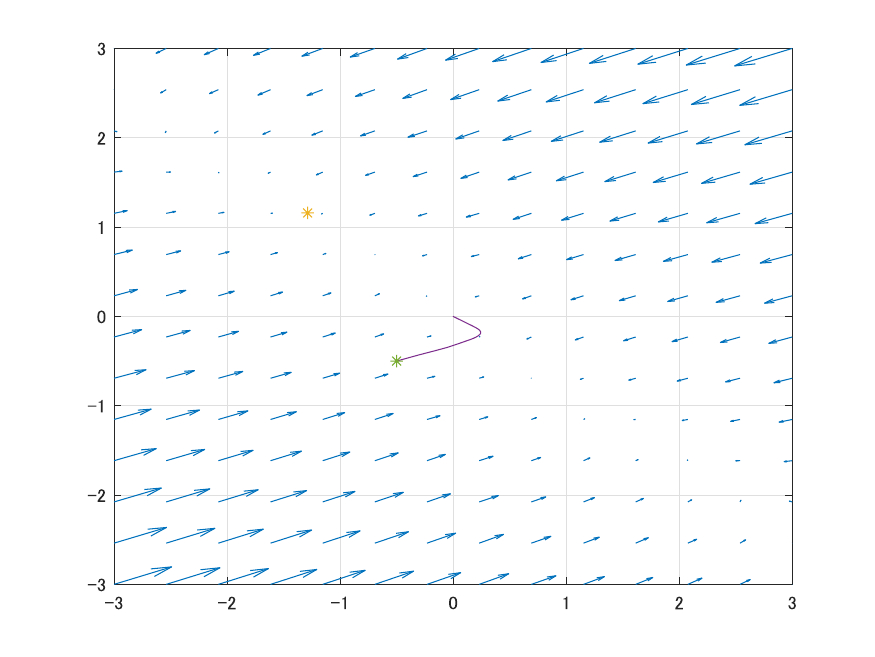}
    \vspace*{-10mm}
    \caption{The state trajectories with initial conditions $x(0)=h_1$ and $x(0)=[-0.5\ -0.5]^T$.}
    \vspace*{-4mm}
    \label{fig:state_trajectory_DD}
\end{figure}
We demonstrate the soundness of Theorem \ref{th:DD}.
Let us consider the case where the coefficient matrices in \eqref{eq:systemG} 
are given by 
\begin{align}
    \scalebox{0.9}{$
        \begin{aligned}
            A=\begin{bmatrix*}[r]
                -0.73 & -0.99 \\
                -0.21 & -0.44
            \end{bmatrix*},
            &
            B=\begin{bmatrix*}[r]
                -1.00 & -0.72 & -0.65 & 0.20 \\
                -0.62 & -0.46 & -0.72 & 0.80
            \end{bmatrix*}, \\
            C=\begin{bmatrix*}[r]
                0.88 & 0.05 \\
                -0.56 & -0.47 \\
                -0.03 & -0.86 \\
                -0.25 & -0.13
            \end{bmatrix*}, 
            &
            D=\begin{bmatrix*}[r]
                -0.65 & 0.92 & 0.41 & 0.54 \\
                -0.95 & 0.52 & 0.29 & -0.54 \\
                0.91 & -0.99 & 0.10 & -0.26 \\
                -0.14 & 0.36 & -0.56 & 0.78
            \end{bmatrix*}. 
        \end{aligned}
        $}
    \notag
\end{align}
For this system, the dual LMI \eqref{eq:ReDLMI_DD} turns out to be feasible, 
and the resulting dual solution $H$ is numerically verified to be $\rank{H}=1$.
The full-rank factorization of $H$ as well as $z^{\ast},w^{\ast}\in\bb{R}^4$ 
in Theorem \ref{th:DD} are
\begin{align}
    \scalebox{0.7}{$
        h_1=\begin{bmatrix*}[r]
            -1.2876 \\
            1.1584
        \end{bmatrix*}, \ 
        h_2=\begin{bmatrix*}[r]
            -0.1118 \\
            0.0000 \\
            -0.1118 \\
            0.1118
        \end{bmatrix*}, \ 
        z^{\ast}=\begin{bmatrix*}[r]
            -0.9880 \\
            0.1900 \\
            -1.0996 \\
            0.3368
        \end{bmatrix*}, \ 
        w^{\ast}=\begin{bmatrix*}[r]
            -0.1118 \\
            0.0000 \\
            -0.1118 \\
            0.1118
        \end{bmatrix*}
        \left( =h_2 \right).
        $}
    \notag
\end{align}
It is clear that $h_1 \ne 0$ (the assertion (i) Theorem \ref{th:DD}). 
Fig. \ref{fig:phi_wc_DD} shows the input-output map 
of the destabilizing nonlinear (piecewise linear) operator $\phi_{\mrm{wc}}$ 
constructed by following the procedure in the preceding subsection. 
From Fig. \ref{fig:phi_wc_DD}, 
we can readily see 
$\phi_{\mrm{wc}}$ is odd and $\phi_{\mrm{wc}}\in\slope{0}{1}$ (the assertion (ii)). 
Fig. \ref{fig:state_trajectory_DD} shows the vector field 
of system $\Sigma$ with the nonlinearity $\Phi=\Phi_{\mrm{wc}}$, 
together with the state trajectories from initial states $x(0)=h_1$ and $x(0)=[-0.5 \ -0.5]^T$.
The state trajectory from the initial state $x(0)=[-0.5 \ -0.5]^T$ converges to the origin. 
However, as proved in Theorem \ref{th:DD}, 
the state trajectory from the initial state $x(0)=h_1$ does not evolve and 
$x(0)=h_1$ is confirmed to be an equilibrium point of the system $\Sigma$ (the assertion (iii)).

%%%%%%%%%%%%%%%%%%%%%%%%%%%%%%%%%%%%%%%%%%%%%%%%%%%%%%%%%%%%%%%%%%%%%%%%%%%%%%%%%%%%%%%%%%%%%%%%%%%%%%%%%%%%%%%%%%%%%%%%%%%%%%%%%%%%%%%%%%%%%%%%%%%%%%%%%%%%%
\section{CONCLUSIONS}
For the absolute stability analysis of nonlinear feedback systems with slope-restricted nonlinearities, 
we investigated the dual of LMIs formulated with static OZF multipliers in the framework of IQC. 
As the main result, we showed that if the dual solution satisfies a certain rank condition, 
then we can extract a nonlinearity that destabilizes the target feedback system 
from the assumed class of slope-restricted nonlinearities. 
Future topics include the extension of this result 
to the absolute-stability analysis problem of feedback systems with slope-restricted and idempotent nonlinearities 
to which novel sets of OZF-type multipliers have been recently obtained \cite{Yuno_CDC2024}.  

%%%%%%%%%%%%%%%%%%%%%%%%%%%%%%%%%%%%%%%%%%%%%%%%%%%%%%%%%%%%%%%%%%%%%%%%%%%%%%%%

%%%%%%%%%%%%%%%%%%%%%%%%%%%%%%%%%%%%%%%%%%%%%%%%%%%%%%%%%%%%%%%%%%%%%%%%%%%%%%%%

%%%%%%%%%%%%%%%%%%%%%%%%%%%%%%%%%%%%%%%%%%%%%%%%%%%%%%%%%%%%%%%%%%%%%%%%%%%%%%%%
%\section*{APPENDIX}


%\section*{ACKNOWLEDGMENT}

%%%%%%%%%%%%%%%%%%%%%%%%%%%%%%%%%%%%%%%%%%%%%%%%%%%%%%%%%%%%%%%%%%%%%%%%%%%%%%%%

\end{document}